\documentclass[12pt]{article}
\usepackage{amsmath,amsthm,amscd,amsfonts}

\newcounter{num}
\newtheorem{theorem}{Theorem}[section]

\theoremstyle{plain}

\newtheorem{prop}[theorem]{Proposition}
\newtheorem{lemma}[theorem]{Lemma}

\newtheorem*{ThA}{Theorem~A}

\newtheorem*{assum}{Assumption}

\theoremstyle{remark}

\theoremstyle{definition}

\newtheorem*{question}{Question}

\DeclareMathOperator{\rank}{rank}
\DeclareMathOperator{\diag}{diag}

\numberwithin{equation}{section} 

\begin{document}

\baselineskip 22pt plus 2pt

\newcommand{\be}        {\begin{eqnarray}}
\newcommand{\ee}        {\end{eqnarray}}
\newcommand{\pl}{\partial}
\newcommand{\sbs}{\subset}
\newcommand{\vr}{\varphi}
\newcommand{\lm}{\lambda}
\newcommand{\eps}{\varepsilon}
\newcommand{\nb}{\nabla}
\newcommand{\wt}{\widetilde}

\newcommand{\cV}      {{\cal V}}
\newcommand{\cE}      {{\cal E}}
\newcommand{\cD}      {{\cal D}}
\newcommand{\cK}      {{\cal K}}
\newcommand{\cB}      {{\cal B}}
\newcommand{\cI}      {{\cal I}}
\newcommand{\cM}      {{\cal M}}
\newcommand{\cA}      {{\cal A}}
\newcommand{\cR}      {{\cal R}}
\newcommand{\cT}      {{\cal T}}
\newcommand{\cP}      {{\cal P}}
\newcommand{\cC}      {{\cal C}}
\newcommand{\cQ}      {{\cal Q}}
\newcommand{\cG}      {{\cal G}}
\newcommand{\cW}      {{\cal W}}
\newcommand{\cL}      {{\cal L}}
\newcommand{\cF}      {{\cal F}}
\newcommand{\cS}      {{\cal S}}
\newcommand{\cH}      {{\cal H}}
\newcommand{\cN}      {{\cal N}}
\newcommand{\AAA}       {{\Bbb A}}     
\newcommand{\aaa}       {{\LBbb A}}
\newcommand{\BB}        {{\Bbb B}}
\newcommand{\bb}        {{\LBbb B}}
\newcommand{\CC}        {\mathbb{C}}
\newcommand{\cc}        {{\LBbb C}}
\newcommand{\DD}        {{\Bbb D}}
\newcommand{\dd}        {{\LBbb D}}
\newcommand{\EEE}       {{\Bbb E}}
\newcommand{\eee}       {{\LBbb E}}
\newcommand{\FF}        {\mathbb{F}} 
\newcommand{\F}         {\mathbb{F}} 
\newcommand{\ff}        {{\LBbb F}}
\newcommand{\GGG}       {{\Bbb G}}
\newcommand{\ggg}       {{\LBbb G}}
\newcommand{\HH}        {{\Bbb H}}
\newcommand{\hh}        {{\LBbb H}}
\newcommand{\II}        {{\Bbb I}}
\newcommand{\ii}        {{\LBbb I}}
\newcommand{\JJ}        {{\Bbb J}}
\newcommand{\jj}        {{\LBbb J}}
\newcommand{\KK}        {{\Bbb K}}
\newcommand{\kk}        {{\LBbb k}}
\newcommand{\LLL}       {{\Bbb L}}
\newcommand{\lll}       {{\LBbb L}}
\newcommand{\MM}        {{\Bbb M}}
\newcommand{\mm}        {{\LBbb M}}
\newcommand{\NN}        {\mathbb{N}}
\newcommand{\nn}        {{\LBbb N}}
\newcommand{\OO}        {{\Bbb O}}
\newcommand{\oo}        {{\LBbb O}}
\newcommand{\PP}        {{\Bbb P}}
\newcommand{\pp}        {{\LBbb P}}
\newcommand{\QQ}        {{\Bbb Q}}
\newcommand{\qq}        {{\LBbb Q}}
\newcommand{\RR}       { \mathbb{R}}
\newcommand{\rr}        {{\LBbb R}}
\newcommand{\SSS}       {{\Bbb S}}
\newcommand{\sss}       {{\LBbb S}}
\newcommand{\TTT}       {{\Bbb T}}
\newcommand{\ttt}       {{\LBbb t}}
\newcommand{\UU}        {{\Bbb U}}
\newcommand{\VV}        {{\Bbb V}}
\newcommand{\vv}        {{\LBbb V}}
\newcommand{\WW}        {{\Bbb W}}
\newcommand{\ww}        {{\LBbb W}}
\newcommand{\XX}        {{\Bbb X}}
\newcommand{\xx}        {{\LBbb X}}
\newcommand{\YY}        {{\Bbb Y}}
\newcommand{\yy}        {{\LBbb Y}}
\newcommand{\ZZ}        {\mathbb{Z}}
\newcommand{\zz}        {{\LBbb Z}}
\newcommand{\x}{\mathbf{x}}
\newcommand{\y}{\mathbf{y}}
\newcommand{\z}{\mathbf{z}}
\newcommand{\1}{\mathbf{1}}
\newcommand{\rO}{\mathrm{O}}
\newcommand{\rS}{\mathrm{S}}
\newcommand{\gl}{\mathbf{GL}}
\newcommand{\lan}{\langle}
\newcommand{\ran}{\rangle}
\newcommand{\an}[1]{\lan#1\ran}
\newcommand{\mr}{\mathrm{mr}}
\newcommand{\trans}{^\top}
\title{On the minimum rank of a graph over finite fields}
\author{{ Shmuel  Friedland and Raphael Loewy}\\ \\
Department of Mathematics, Statistics, and Computer Science,\\
        University of Illinois at Chicago\\
        Chicago, Illinois 60607-7045, USA\\
         \and
Department of Mathematics\\
Technion -- Israel Institute of Technology\\
32000 Haifa, Israel}


\date{June 27, 2010}                                         

\maketitle


\begin{abstract}
In this paper we deal with two aspects of the minimum rank of a simple undirected graph $G$ on $n$
vertices over a finite field $\FF_q$ with $q$ elements, which is denoted by $\mr(\FF_q,G)$.
In the first part of this paper we show that the average minimum rank of simple undirected labeled graphs
on $n$ vertices over $\FF_2$ is $(1-\varepsilon_n)n$, were $\lim_{n\to\infty} \varepsilon_n=0$.

In the second part of this paper we assume that $G$ contains a clique $K_k$ on $k$-vertices.  We show that if $q$ is not a prime
then $\mr(\FF_q,G)\le n-k+1$ for $4\le k\le n-1$ and $n\ge 5$.
It is known that $\mr(\FF_q,G)\le 3$ for $k=n-2$, $n\ge 4$
and $q\ge 4$.  We show that for $k=n-2$ and each $n\ge 10$ there exists a graph $G$ such that $\mr(\FF_3,G)>3$. For $k=n-3$, $n\ge 5$ and $q\ge 4$ we show that $\mr(\FF_q,G)\le 4$.

\end{abstract}


\vspace{1cm}

 \section{Introduction.}
 Let $G=(V,E)$ be a simple undirected graph on the set of vertices $V$  and the set of edges $E$.
 Set $n:=|V|$, and identify $V$ with $\an{n}:=\{1,\ldots,n\}$. Denote by $ij\in  E$ the edge connecting
 the vertices $i$ and $j$.  Let $\FF$ be a field.  For a prime $p$ let $\FF_p$ be the finite field
 of integers modulo $p$.
 Denote by $\rS(\FF,G)$ the set of all symmetric $n\times n$ matrices $A=[a_{ij}]_{i,j=1}^n$ with entries
 in $\FF$, and such that $a_{ij}\ne 0, \ i\ne j$, exactly when $ij\in E$.
 There is no restriction on the main diagonal entries of $A$. Let
 \[
 \mr(\FF,G)=\min\bigl\{\rank A:A\in \rS(\FF,G)  \bigr\}.
 \]
 The problem of determining $\mr(\FF,G)$ has been of significant interest in recent years.

 Here we consider two aspects of this problem.  The first aspect is estimating the average
 of the minimum rank over $\cG_n$, the set of all labeled graphs on $\an{n}$.
 Since the complete graph on $n$ vertices, denoted by
 $K_n=(\an{n},E_n)$ has $n\choose 2$ edges, we have that
 $|\cG_n|=2^{n\choose 2}$.
 Let
 \begin{equation}\label{avminrank}
 \alpha_{n}(\F)=\frac{1}{n2^{n\choose 2}} \sum_{G\in\cG_n} \mr(\F,G)
 \end{equation}
 be the \emph{scaled} $n$-average minimum rank.  It is a very interesting
 problem to estimate $\alpha_n(\F)$ for large $n$, and
 \begin{equation}\label{uplowavmrank}
 \bar\alpha(\F):=\limsup_{n\to\infty} \alpha_n(\F), \quad
 \underline{\alpha}(\F):=\liminf_{n\to\infty}\alpha_n(\F).
 \end{equation}
 In this note we show
 \begin{equation}\label{valminavrankF2}
 \underline{\alpha}(\F_2)=\bar\alpha(\F_2)=1.
 \end{equation}

 In a recent paper Hall, Hogben, Martin and Shader \cite{HHMS} have shown
 that for $n$ sufficiently large
 $0.146907<\alpha_n(\RR)<0.5+\frac{\sqrt{7\ln n}}{n}$.
 The second aspect of this problem is estimating the minimum rank of graphs which contain a clique.
 A {\em $k$-clique} in $G$ is a complete graph $K_k$ that
 occurs as an induced subgraph of $G$.

 We are interested in the following:
 \begin{question}
 Suppose $G$ contains a $k$-clique. When is $\mr(\FF,G)\le n-k+1$?
 \end{question}

 To avoid trivialities, assume that $n=|V|\ge 3$.  It is well known that the inequality $\mr(\FF,G)\le n-1$ always holds. It is also well known that $\mr(\FF,K_n)=1$.  Hence, our question has an affirmative answer for $k=1,2,n$. In fact, the same holds true for $k=3$, by papers of Fiedler \cite{F} and Bento-Leal Duarte~\cite{BD}. Hence we may assume from now on that
 \[
 4\le k\le n-1.
 \]

 The question has also an affirmative answer in case when $\FF$ is an infinite field.
 This appears implicitly in the paper of Johnson, Loewy and Smith \cite{JLS}.  (See \S3.)
 In her M.Sc. thesis \cite{B} Bank gave an affirmative answer to our question for any finite field $\FF$
 with $|\FF|\ge k-1$. While giving us some information on the finite field case, this result has a drawback,
 namely, for a large $k$, the field is required to be large. However, as we will see here and in subsequent
 sections, at least for the special cases $k=n-1, \ k=n-2$ and $k=n-3$ small fields suffice to get an
 affirmative  answer to our question.
 Thus, the case of a finite field is still of significant interest.
 It follows from Barrett, van der Holst and Loewy \cite{BvdHL} that our question has an
 affirmative answer in case $k=n-1$ and $\FF$ is any field which is not $\FF_2$. As for $\FF_2$,
 let $n\ge 5$ and consider the graph $G$ obtained from $K_{n-1}$ by adding a new vertex $v$ and
 connecting it to exactly two of the vertices of $K_{n-1}$. Then it follows from \cite{BvdHL} that
 $\mr(\FF_2,G)=3$. Thus, over $\FF_2$ our question does not have an affirmative answer in all cases.

 We now briefly survey the contents of this paper.
 In Section~2 we prove the equality (\ref{valminavrankF2});
 in Section~3 we deal briefly with the case of an infinite field;
 in Section~4 we deal with the case
 $k=n-2$; in Section~5 we give an affirmative answer to our question for any finite field $\FF$ which is
 not a prime field; in Section~6 we consider the case
 $k=n-3$.
 \section{Scaled average minimum rank over $\FF_2$}
 We first recall some known results on certain classes of matrices $\F_2^{n\times n}$.
 \begin{lemma}\label{cardof2} Let $\rO(n,\F_2)$ be the orthogonal group of $n\times n$ matrices.
 Then $O(n):=|\rO(n,\F_2)|$ is equal to $C(n)2^{\frac{n(n-1)}{2}}$,
 where
 \begin{equation}\label{numorthmatrices}
 C(1)=C(2)=1,  \textrm{ and }
 C(n)=\prod_{i=1}^{\lfloor \frac{n-1}{2}\rfloor} (1-2^{-2i})   \textrm{ for } n>2.
 \end{equation}
 \end{lemma}
 See for example \cite[p. 158]{MW69}. (Note that the formula in \cite{MW69} has a different equivalent form.)

 Denote by $\rS(n,\F_2)\subset \F_2^{n\times n}$ the subspace of symmetric matrices.
 \begin{lemma}\label{canformsymf2}  Let $A\in\rS(n,\F_2)$ have rank $k$.  Then the $A$ has the
 following form.
 \begin{enumerate}
 \item\label{oddk}
 If $k$ is odd then $A=XX\trans$, where $X\in\F_2^{n\times k}$
 and $\rank X=k$.
 \item\label{evenk}
 If $k=2l$ is even then $A$ has two possible nonequivalent
 forms.  First $A=XX\trans$, where $X\in\F_2^{n\times k}$
 and $\rank X=k$.  Second $A=X(\oplus_{j=1}^l H_2) X\trans$, where $X\in\F_2^{n\times
 k}$, $\rank X=k$ and $H_2:=\left[\begin{array}{cc} 0&1\\1&0\end{array}\right]$.

 \end{enumerate}
 \end{lemma}
 See for example \cite[Thm 2.6]{Fri91}.
 Let $J_{2n}\in\rS(2n,\F_2)$ be the direct sum of $n$ copies of $H_2$.
 Denote by
 $\textrm{Sym}(2n,\F_2):=\{T\in \F^{2n\times 2n}, T\trans J_{2n} T=J_{2n}\}$ the symplectic group
 of order $2n$ over $\F_2$.
 \begin{lemma}\label{cardsym2nf2}  The cardinality of $\textrm{Sym}(2n,\F_2)$ is given by $O(2n+1)$.
 \end{lemma}
 See \cite[p'6, p'11]{Car85}.
 The following result is probably well known, and we bring its proof for completeness.
 \begin{lemma}\label{nkmatrrankk} The number of
 $n\times k$ matrices $X\in\F_2^{n\times k}$ of rank $k\le n$ is equal to
 \begin{equation}\label{nkmatrrankk1}
 N(n,k)=(2^n-1)(2^n-2)\ldots(2^n-2^{k-1})=2^{nk}(1-\frac{1}{2^n})
 \ldots (1-\frac{1}{2^{n-k+1}})
 \end{equation}
 \end{lemma}
 \proof  The proof is by induction on $k$.  For $k=1$, $A$ can have any first column,
 except the zero column.  Hence $N(n,1)=2^n-1$.  Assume that the number of
 $n\times k$ matrices $X\in\F_2^{n\times k}$ of rank $k\le n$ is equal to $N(n,k)$ for $k\le n-1$.
 Observe that $A\in\F_2^{n\times(k+1)}$ has rank $k+1$ if and only if the first $k$ columns of $A$ are linearly
 independent, and the last column is not a linear combination of the first $k$ columns.  Assume that the first $k$ columns  of $A$ are linearly independent.
 Then the number of vectors, which are linear combinations of the first $k$ columns of $A$ are $2^k$.  Hence the last column of $A$ can be chosen in $2^n-2^k$ ways such that rank $A=k+1$. Thus $N(n,k+1)=N(n,k)(2^n-2^k)$.  \qed
 \begin{lemma}\label{prodest} For $n\ge 2$
 $$1>\prod_{j=1}^{n-1}(1-\frac{1}{2^j})>\prod_{j=1}^{\infty}(1-\frac{1}{2^j})>\frac{1}{4}.$$
 \end{lemma}
 \proof
 Recall that
 $\log(1-x)=-\sum_{m=1}^{\infty} \frac{x^m}{m} \textrm{ for } x\in (-1,1)$.
 \begin{eqnarray*}
 \sum_{j=1}^{n-1}\log(1-\frac{1}{2^j})> \sum_{j=1}^{\infty}
 \log(1-\frac{1}{2^j})=-\sum_{j=1}^{\infty}\sum_{m=1}^{\infty}\frac{1}{m}\frac{1}{2^{jm}}=
 -\sum_{m=1}^{\infty}\frac{1}{m}\sum
 _{j=1}^{\infty}\frac{1}{2^{jm}}=\\-\sum_{m=1}^{\infty}\frac{1}{m}\frac{1}{2^m(1-2^{-m})}>
 -\sum_{m=1}^{\infty}\frac{1}{m} \frac{2}{2^m}=2\log(1-\frac{1}{2})=\log \frac{1}{4}.
 \end{eqnarray*}
 \qed

 Combine Lemmas \ref{cardof2}, \ref{nkmatrrankk} and \ref{prodest} to deduce
 \begin{equation}\label{nknonest}
 1\ge C(n)> \frac{1}{4}, \quad 2^{nk}> N(n,k)>2^{nk-2}.
 \end{equation}
 \begin{lemma}\label{estsymrkmat}  Let $\theta(n,k)$ be the number of $A\in\rS(n,\F_2)$ of rank $k$.
 If $k$ is odd then $\theta(n,k)=\frac{N(n,k)}{O(k)}$.  If $k$ even then then
 $\theta(n,k)=\frac{N(n,k)}{O(k)}+\frac{N(n,k)}{O(k+1)}$.  In particular
 \begin{equation}\label{thetbounds}
 2^{(n-\frac{k-1}{2})k-2}<\theta(n,k)<2^{(n-\frac{k-1}{2})k+3}.
 \end{equation}
 \end{lemma}
 \proof  We first find the number of distinct matrices $A\in \rS_n(\F_2)$ of rank $k$ of
 the form $XX\trans$, where $X\in \F^{n\times k}$ of rank $k$.  Assume that $X,Y\in \F_2^{n\times k}$
 have both rank $k$ and $A=XX\trans=YY\trans$.  Since the columns of $X$ and $Y$ form a basis of
 the column space of $A$,
 it follows that $Y=XQ$ for some $Q\in\gl(k,\F_2)$.  We claim that $Q\in\rO(k,\F_2)$.  Indeed, since the
 columns of $X$ are linearly independent, they can be extended to a basis of $\F_2^n$.
 Hence, there exists $Z\in \F_2^{k\times n}$ such that $ZX=I_k$.  So
 $$I_k=Z(XX\trans)Z\trans=Z(YY\trans)Z\trans =QQ\trans.$$
 Hence the number of such symmetric matrices of rank $k$ is $\frac{N(n,k)}{O(k)}$.
 Use Lemmas  \ref{cardof2}, (\ref{numorthmatrices}) and (\ref{nknonest}) to conclude that
 $$ 2^{(n-\frac{k-1}{2})k-2}<\frac{N(n,k)}{O(k)}<2^{(n-\frac{k-1}{2})k+2}.$$
 If $k$ is odd, we are done in view of Lemma \ref{canformsymf2}.
 Suppose that $k$ is even.  Then Lemma \ref{canformsymf2} claims that we have a second kind of $A\in\rS(n,\F_2)$ of rank $k$, which is of the form $XJ_kX\trans$, where $X\in\F_2^{n\times k}$
 is a matrix of rank $k$.  As  in the first case, $A=XJ_kX\trans=YJ_kY\trans$ if and only if $Y=XP$
 where $P$ is a symplectic matrix.  In view of Lemma \ref{cardsym2nf2} the cardinality of the symplectic
 group over $\F_2$ of order $k$ is $O(k+1)$.  Hence the number the symmetric matrices of order $n$,
 rank $k$ of the second kind is $\frac{N(n,k)}{O(k+1)}$, which is less than $\frac{N(n,k)}{O(k)}$.  In particular (\ref{thetbounds}) always holds.
 \qed
 \begin{theorem}\label{avmrnk2}  Let $\alpha_n(\F_2)$ be the scaled $n$-average minimum rank
 over all simple graphs on $n$ vertices over the field $\F_2$, as defined by (\ref{avminrank}).
 Define $\underline{\alpha}(\F_2)\le \overline{\alpha}(\F_2)$ as in (\ref{uplowavmrank}).
 Then $\underline{\alpha}(\F_2)= \overline{\alpha}(\F_2)=1$.
 \end{theorem}
 \proof
 Let us estimate the number of all graphs whose minimum rank
 is at most $k$.  This number is at most the number of symmetric matrices
 in $\rS_n(\F_2)$ whose rank is at most $k$.
 (In other words, we assume the optimal condition that for each
 graph of minimum rank $r\le k$ there is only one matrix of
 rank $r$ and all other matrices are of rank greater than $r$.)
 Lemma \ref{estsymrkmat} yields that
 the upper bound on this number is
 $$\sum_{r=1}^k 8\cdot 2^{(n-\frac{r-1}{2})r}< 8\cdot 2^{(n-\frac{k-1}{2})k}
 \sum_{j=0}^{\infty} \frac{1}{2^j}=16\cdot 2^{(n-\frac{k-1}{2})k}.$$

 The number of graphs is $2^{\frac{n(n-1)}{2}}$.
 Fix $t \in (0,1)$.  Suppose that $k\le nt$.
 Then the number of all graphs with rank at most $tn$ is less than
 $16\cdot 2^{\frac{(2-t)t n^2}{2}}$.  Note that
 $$\lim_{n\to\infty}\frac{(tn) 16\cdot 2^{\frac{(2-t)t
 n^2}{2}}}{2^{\frac{n(n-1)}{2}}}=0.$$
 So the contribution of all these graphs to the average is
 zero.  Thus all the contribution to $\underline{\alpha}(\F_2)$
 comes from the graphs whose minimum rank is greater then $tn$
 for any $t\in (0,1)$.  Hence, $\underline{\alpha}(\F_2)\ge t$.
 Thus we showed (\ref{valminavrankF2}).  \qed

 \section{Infinite fields }\label{infield}
 The question raised in the introduction has an affirmative answer in case $\F$ is an infinite
 field.  This appears implicitly in \cite{JLS}.
 For the sake of clarity we state this result and give a short sketch of its proof.
 \begin{ThA}
 Let $n$ and $k$ be positive integers such that $4\le k<n$, and let $G$ be a graph on $n$ vertices which contains a $k$-clique as an induced subgraph. Then, for any infinite field $F$
 \[
 \mr(\FF,G)\le n-k+1.
 \]
 \end{ThA}

 \begin{proof}
 We use the first part in the appendix of \cite{JLS}, and in particular Observation~A.1, Lemma~A.2 and the discussion between these two results.

 We can assume that $1,2,\dots,k$ are vertices of a $k$-clique in $G$. Given any $i,j$ in
 $\{1,2,\dots,k-1\}$ with $i\ne j$, there is a path of length two in $G$ from $i$ to $j$, namely
 the path whose only intermediate vertex is $k$. The conditions of Lemma~A.2 in \cite{JLS} are satisfied.
 Hence there exists a matrix $A\in \rS(\FF,G)$ of the form
 \[
 A=
 \begin{bmatrix}
 C & A_{12}\\
 A\trans_{12} & A_{22}
 \end{bmatrix},
 \]
 where $A_{22}$ is an invertible $n-k+1\times n-k+1$ matrix, and $C=A_{12}A^{-1}_{22}A\trans_{12}$.
 Then $\rank A=n-k+1$, implying $\mr(\FF,G)\le n-k+1$.
 \end{proof}

 In light of Theorem~A, we can assume from now on that $F$ is {\em  a finite field}.

 \section{The case $k=n-2$.}
 In \cite[Proposition~4.2.3]{B} the following result has been proved.

 \begin{prop}\label{P2.1}
 Let $G$ be a graph on $n\ge 4$ vertices and suppose that $G$ contains $K_{n-2}$ as an induced subgraph.
 Then $\mr(\FF,G)\le 3$ for any field $\FF$ with $|\FF|>3$.
 \end{prop}

 It is known that Proposition~\ref{P2.1} is not always valid when $\FF=\FF_2$.  Examples of such graphs are
 given in~\cite{BGL}, for example graph $\# 14$ there, whose minimum rank over $\FF_2$ is 4.

 \vspace{0.5cm}

 The field $\FF_3$ is not discussed in~\cite{B}. It is our purpose to show that Proposition~\ref{P2.1}
 is not always valid when $\FF=\FF_3$, so in fact, it is best possible.

 \begin{theorem} For every $n$ such that $n\ge 10$
 there exists a graph $G$ on $n$ vertices containing $K_{n-2}$ as an induced subgraph, and such that
 $\mr(\FF_3,G)>3$.
 \end{theorem}

 \begin{proof}
 Let $G$ be a graph containing $K_{n-2}$ as an induced subgraph. We label the vertices of $G$ so that
 $1,2,\dots,n-2$ are the vertices of an $(n-2)$-clique of $G$. We assume that $n-2,n-1$ and $n$ are
 independent vertices of $G$, that is no two of them are adjacent. Let $A\in \rS(\FF_3,G)$ and partition
 $A$ as follows:
 \[
 A=
 \begin{bmatrix}
 A_{11} & A_{12}\\
 A\trans_{12} & A_{22}
 \end{bmatrix},
 \]
 where $A_{22}$ is $3\times 3$. Note that each row vector of $A_{12}$ is in $\FF^3_3$, and its first entry
 is nonzero, because vertex $n-2$ belongs to the given $(n-2)$-clique. Applying, if necessary, a suitable
 congruence transformation $D\trans AD$, where $D$ is an invertible diagonal matrix, we may assume without loss
 of generality that {\em the first entry} of each row vector of $A_{12}$ is~1.

 It follows that every row vector of $A_{12}$ has one of the following four patterns,\\
 \hspace*{2cm} (i)~~$(1,0,0)$; \  (ii)~~$(1,*,0)$; \ (iii)~~$(1,0,*)$; \ (iv)~~$(1,*,*)$,\\
 where $*$ denotes a nonzero element of $\FF_3$. We add now the following assumption on $G$ (expressed in terms of $A$):
 \begin{assum}
 The matrix $A_{12}$ has at least one row of the pattern (i), and at least two distinct rows of each of the other patterns.
 \end{assum}

 Observe first that each column of $A_{12}$ is not a zero column.
 Suppose that $\mr(\FF_3,G)\le 3$, and let $A\in \rS(\FF_3,G)$ be such that $\rank A=mr(\FF_3,G)$. Consider $A_{22}$.
 The independence of vertices $n-2,n-1$ and $n$ implies that $A_{22}=\diag(\alpha,\beta,\gamma)$, where $\alpha,\beta,\gamma\in\FF_3$. We claim that $A_{22}$ is invertible.

 Indeed, $A_{12}$ contains the row vectors $(1,*,*),(1,0,*)$ and $(1,0,0)$, and suppose that they are the $i$th, $j$th and $k$th rows of $A_{12}$. Since
 $
 \Bigl[
 \begin{smallmatrix}
 1 & * & *\\
 1 & 0 & *\\
 1 & 0 & 0
 \end{smallmatrix}
 \Bigr]
 $
 is invertible, it follows that $\rank A\ge 3$, so we must have $\rank A=3$. Moreover, rows $i,j$ and $k$ of $A$ span the row space of $A$. In particular, each of the nonzero rows $n-2,n-1$ and $n$ of $A$ is a linear combination of rows $i,j,$ and $k$, and hence $\alpha,\beta,\gamma\ne 0$, so $A_{22}$ is invertible.

 Let $E=A_{22}$. Since $\rank A=\rank E=3$, it follows, using Schur complement, that
 \[
 A_{11}=A_{12}E^{-1}A\trans_{12},
 \]
 and, in particular, all off-diagonal entries of $A_{12}E^{-1}A\trans_{12}$ must be nonzero.
 We will assume from now on that $\alpha=1$, as we can replace $E$ by $2E$ if necessary, so that
 $E=\diag(1,\beta,\gamma)$.

 We  will make repeated use of the requirement that all off-diagonal elements of
 $A_{12}E^{-1}A\trans_{12}$ must be nonzero. So if $x=(1,x_2,x_3)$ and $y=(1,y_2,y_3)$ are
 row vectors of any $2$ distinct rows of $A_{12}$ then
 \be\label{Eq2.1}
 1+\beta^{-1}x_2y_2+\gamma^{-1}x_3y_3=1+\beta x_2y_2+\gamma x_3y_3\ne 0.
 \ee

 Recall that $x^2=1$ for any $x\in\FF_3\backslash\{0\}$.  Furthermore for $x,y\in \FF_3\backslash\{0\}$ $1+xy\ne 0\iff x=y$.
 We distinguish four cases.

 \underline{Case~1.}~~$\beta=\gamma=1$, so $E=I_3$.
 Assume that $(1,x,0),(1,0,y)$ are two rows of the pattern (ii) and (iii) respectively.
 Let $(1,z_1,z_2)$ be a row of the pattern (iv).  \eqref{Eq2.1} for the pairs $(1,x,0),\z$ and $(1,0,y),\z$ yield that $\z=(1,x,y)$.
 Since there are at least two distinct rows of the form $(1,x,y)$ we contradict \eqref{Eq2.1}.

 \underline{Case~2.}~~$\beta=1,\gamma=2$, so $E=E^{-1}=\diag(1,1,2)$.

 Consider the rows of $A_{12}$ with pattern $(1,0,*)$. By~\eqref{Eq2.1}, no two of them can be $(1,0,1)$ and no two of them can be $(1,0,2)$.
 Hence $A_{12}$ has exactly two rows $(1,0,1)$ and $(1,0,2)$.  Let $\z=(1,z_1,z_2)$ be a row of the pattern (iv).
 Then \eqref{Eq2.1} cannot hold for the two pairs $(1,0,1),\z$ and $(1,0,2),\z$.

 \underline{Case~3.}~~$\beta=2,\gamma=1$, so $E=E^{-1}=\diag(1,2,1)$.

 The contradiction is obtained as in Case~2.

 \underline{Case~4.}~~$\beta=\gamma=2$, so $E=E^{-1}=\diag(1,2,2)$.

 The contradiction is obtained as in Case~2.  This concludes the proof of the theorem.
 \end{proof}

 \section{The case of a finite non-prime field.}

 In this section we assume that $\FF$ is any finite non-prime field. We prove
 \begin{theorem}\label{Th3.1}
 Let $G$ be a graph on $n\ge 5$ vertices, and let $4\le k\le n-1$. Suppose that $G$ contains $K_k$
 as an induced subgraph. Let $\FF$ be a finite non-prime field of characteristic $p$. Then
 $\mr(\FF,G)\le n-k+1$.
 \end{theorem}

 \begin{proof}
 The assumption on $\FF$ implies that $\FF$ is a finite extension of $\FF_p$, and $\FF\ne\FF_p$.
 We label the vertices of $G$ so that $1,2,\dots,k$ are the vertices of a $k$-clique. Let $A\in \rS(\FF,G)$ and
 partition $A$ be as follows:
 \[
 A=
 \begin{bmatrix}
 A_{11} & A_{12}\\
 A\trans_{12} & A_{22}
 \end{bmatrix},
 \]
 where $A_{11}$ is $k\times k$. Let $H$ be the subgraph of $G$ induced by vertices $k+1,k+2,\dots,n$.
 It is straightforward to see that there exists $B\in \rS(\FF_p,H)$ which is invertible. Indeed, let every
 nonzero off-diagonal entry of $B$ be 1, and let $b_{11}=1$. Then we can choose $b_{22}\in\FF_p$ so that
 the principal minor of order~2 in the top left corner is~1. Similarly, we can sequentially choose
 $b_{33},\dots,b_{n-k,n-k}$ in $\FF_p$ so that all leading principal minors of $B$ are~1. We now pick
 every nonzero entry of $A_{12}$ to be~1, and we let $A_{22}=\beta^{-1}B$, where $\beta\in F/\FF_p$.
 Hence $A^{-1}_{22}=\beta B^{-1}$, where all entries of $B^{-1}$ are in $\FF_p$.

 We claim that $A_{11}$ can be chosen so that $A_{11}\in \rS(\FF,K_k)$ and so that
 \[
 A_{11}-A_{12}A^{-1}_{22}A\trans_{12}=J_k,
 \]
 where $J_k$ is the all ones $k\times k$ matrix. Indeed, let
 \[
 A_{11}=J_k+A_{12}A^{-1}_{22}A\trans_{12}=J_k+\beta A_{12}B^{-1}A\trans_{12}.
 \]
 Every off-diagonal element of $J_k+\beta A_{12}B^{-1}A\trans_{12}$ is of the form $1+\beta a, \ a\in\FF_p$,
 and so must be nonzero. For our choice of $A_{11},A_{12}$ and $A_{22}$, the Schur complement of
 $A_{22}$ in $A$ has rank one, so $\rank A=n-k+1$. Hence $\mr(\FF,G)\le n-k+1$.
 \end{proof}

 \section{The case $k=n-3$.}

 In this section we consider the case $k=n-3$, that is, we assume that $G$ contains an $(n-3)$-clique. We prove
 \begin{theorem}\label{Th4.1}
 Let $G$ be a graph on $n\ge 5$ vertices, and suppose that $G$ contains $K_{n-3}$ as an induced subgraph.
 Then, $\mr(\FF,G)\le 4$ for every field $\FF$ with $|\FF|>3$.
 \end{theorem}

 \begin{proof}
 We assume that $\FF$ is a finite field with $|\FF|>3$. Let us label the vertices of $G$ so that
 $1,2,\dots,n-3$ are the vertices of an $(n-3)$-clique in $G$. Let $A\in \rS(\FF,G)$ and partition $A$
 as follows
 \[
 A=
 \begin{bmatrix}
 A_{11} & A_{12}\\
 A\trans_{12} & A_{22}
 \end{bmatrix},
 \]
 where $A_{22}$ is $3\times 3$. Our goal is to show that $A$ can be chosen so that $A_{22}$ is invertible and its Schur complement in $A$ is $J_{n-3}$. Then it follows for this $A$ that $\rank A=4$, implying $\mr(\FF,G)\le 4$.

 Denote by $\1$ the all ones vector (of order that should be clear from the discussion).
 Each column vector of $A\trans_{12}$ must be of one of the following eight patterns:
 \[
 \begin{bmatrix}
 0\\0\\0
 \end{bmatrix},
 \begin{bmatrix}
 *\\0\\0
 \end{bmatrix},
 \begin{bmatrix}
 0\\ * \\0
 \end{bmatrix},
 \begin{bmatrix}
 0\\0\\ *
 \end{bmatrix},
 \begin{bmatrix}
 *\\ *\\0
 \end{bmatrix},
 \begin{bmatrix}
 *\\0\\ *
 \end{bmatrix},
 \begin{bmatrix}
 0\\ *\\ *
 \end{bmatrix},
 \begin{bmatrix}
 *\\ *\\ *
 \end{bmatrix},
 \]
 where $*$ denotes a nonzero entry. We may assume without loss of generality that the columns of
 $A\trans_{12}$ are already arranged in eight groups (although not all  of them must be present) according
 to these patterns. Moreover, we pick the nonzero entries of $A\trans_{12}$ so that all columns with the
 same pattern are equal. In fact, we pick the entries of $A\trans_{12}$ so that
 \[
 A\trans_{12}=
 \begin{bmatrix}
 0 & \alpha_1\1\trans & 0 & 0 & \alpha_4\1\trans & \alpha_6\1\trans & 0 & \alpha_{10}\1\trans\\
 0 & 0 & \alpha_2\1\trans & 0 & \alpha_5\1\trans & 0 & \alpha_8\1\trans & \alpha_{11}\1\trans\\
 0 & 0 & 0 & \alpha_3\1\trans & 0 & \alpha_7\1\trans & \alpha_9\1\trans & \alpha_{12}\1\trans
 \end{bmatrix}
 \in \FF^{3,n-3},
 \]
 where $\alpha_1,\alpha_2,\dots,\alpha_{12}$ will be determined later, and not all block columns must
 be present.

 We will distinguish four cases, according to the pattern of the entries of $A_{22}$. In each case we
 pick the entries of $A_{12}$ and $A_{22}$ so that $A_{22}$ is nonsingular and so that all off-diagonal
 entries of $J_{n-3}+A_{12}A^{-1}_{22}A\trans_{12}$ are nonzero. Then we let
 $A_{11}=J_{n-3}+A_{12}A^{-1}_{22}A\trans_{12}$, and we get what we want.

 \underline{Case~1.}~~Suppose that $A_{22}$ is a diagonal matrix. we let
  $A_{22}=\beta^{-1}I_3$, where $0\ne\beta\in \FF$. Hence $A^{-1}_{22}=\beta I_3$.
  A straightforward computation yields
 \[
 A_{12}A^{-1}_{22}=\beta
 \begin{bmatrix}
 0 & 0 & 0 \\
 \alpha_1\1 & 0 & 0\\
 0 & \alpha_2\1 & 0\\
 0 & 0 & \alpha_3\1\\
 \alpha_4\1 & \alpha_5\1 & 0\\
 \alpha_6\1 & 0 & \alpha_7\1\\
 0 & \alpha_8\1 & \alpha_9\1\\
 \alpha_{10}\1 & \alpha_{11}\1 & \alpha_{12}\1
 \end{bmatrix},
 \]
 so all off-diagonal entries of $A_{12}A^{-1}_{22}A\trans_{12}$ are contained in
 \be
 &&\Bigl\{0,\beta\alpha^2_1,\beta\alpha_1\alpha_4,\beta\alpha_1\alpha_6,\beta\alpha_1\alpha_{10},
 \beta\alpha^2_2,\beta\alpha_2\alpha_5,\beta\alpha_2\alpha_8,\beta\alpha_2\alpha_{11},\beta\alpha^2_3,\beta\alpha_3\alpha_7,\notag\\
 &&\beta\alpha_3\alpha_9,\beta\alpha_3\alpha_{12},\beta\bigl(\alpha^2_4+\alpha^2_5\bigr),
 \beta\alpha_4\alpha_6,\beta\alpha_5\alpha_8,\beta\bigl(\alpha_4\alpha_{10}+\alpha_5\alpha_{11}\bigr),
 \notag\\
 &&\beta\bigl(\alpha^2_6+\alpha^2_7 \bigr),\beta\alpha_7\alpha_9,\beta(\alpha_6\alpha_{10}+\alpha_7\alpha_{12}),
 \beta\bigl(\alpha^2_8+\alpha^2_9\bigr),
 \beta\bigl(\alpha_8\alpha_{11}+\alpha_9\alpha_{12}\bigr),\notag\\
 &&\beta\bigl(\alpha^2_{10}+\alpha^2_{11}+\alpha^2_{12} \bigr)\Bigr\}.\notag
 \ee
 Choose now all $\alpha$'s to be~1. So the constraints on $\beta$ are
 \[
 \beta\ne 0; \ 1+\beta\ne 0; \ 1+2\beta\ne 0; \ 1+3\beta\ne 0.
 \]
 For any $\FF$ with $|\FF|\ge 5$ one can choose $\beta\in F$ to satisfy these constraints.
 For $\FF$ with $|\FF|=4$ one can apply Theorem~\ref{Th3.1} or note that for this field $2=0$ and $3=1$,
 so there are only two constraints.

 \underline{Case~2.}~~Suppose that exactly two of the entries above the main diagonal of $A_{22}$ are zero.
 We assume that they are the elements in the $1,2$ and $1,3$ positions (other possibilities are handled in
 a similar way). Let
 \[
 A_{22}=\beta^{-1}
 \begin{bmatrix}
 1 & 0 & 0\\
 0  & 0 & 1\\
 0 & 1 & 0
 \end{bmatrix}, \ \text{where} \ 0\ne\beta\in \FF, \ \text{so} \ A^{-1}_{22}=\beta
 \begin{bmatrix}
 1 & 0 & 0\\
 0 & 0 & 1\\
 0 & 1 & 0
 \end{bmatrix}.
 \]
 Hence
 \[
 A_{12}A^{-1}_{22}=\beta
 \begin{bmatrix}
 0 & 0 & 0\\
 \alpha_1\1 & 0 & 0\\
 0 & 0 &\alpha_2\1\\
 0 & \alpha_3\1 & 0\\
 \alpha_4\1 & 0 & \alpha_5\1\\
 \alpha_6\1 & \alpha_7\1 & 0\\
 0 & \alpha_9\1 & \alpha_8\1\\
 \alpha_{10}\1 & \alpha_{12}\1 & \alpha_{11}\1
 \end{bmatrix},
 \]
 and a discussion similar to the one in Case~1 shows that if all $\alpha$'s are chosen to
 be~1 then the constraints on $\beta$ are again
 \[
 \beta\ne 0; \ 1+\beta\ne 0; \ 1+2\beta\ne 0; \ 1+3\beta\ne 0.
 \]

 \underline{Case~3.}~~Suppose that exactly one entry above the main diagonal of $A_{22}$ is zero. We assume that it is the $1,2$ entry (other possibilities are handled similarly). Let
 $
 A_{22}=\beta^{-1}\Bigl[
 \begin{smallmatrix}
 0 & 0 & 1\\
 0 & 1 & 1\\
 1 & 1 & 0
 \end{smallmatrix}\Bigr],
 $
 where $0\ne\beta\in \FF$, so $A^{-1}_{22}=\beta\Bigl[
 \begin{smallmatrix}
 &1 & -1 & 1\\
 -&1 & \ 1 & 0\\
 &1 & \ 0 & 0
 \end{smallmatrix}\Bigr].
 $
 Hence
 \[
 A_{12}A^{-1}_{22}=\beta
 \begin{bmatrix}
 0 & 0 & 0\\
 \alpha_1\1 & -\alpha_1\1 & \alpha_1\1\\
 -\alpha_2\1 & \alpha_2\1 & 0\\
 \alpha_3\1 & 0 & 0\\
 (\alpha_4-\alpha_5)\1 & (-\alpha_4+\alpha_5)\1 & \alpha_4\1\\
 (\alpha_6+\alpha_7)\1 & -\alpha_6\1 & \alpha_6\1\\
 (-\alpha_8+\alpha_9)\1 & \alpha_8\1 & 0\\
 (\alpha_{10}-\alpha_{11}+\alpha_{12})\1 & (-\alpha_{10}+\alpha_{11})\1 & \alpha_{10}\1
 \end{bmatrix},
 \]
 so all of the diagonal entries of $A_{12}A^{-1}_{22}A\trans_{12}$ are contained in
 \be
 &&\Bigl\{
 0,\beta\alpha^2_1,-\beta\alpha_1\alpha_2,\beta\alpha_1\alpha_3,\beta\alpha_1(\alpha_4-\alpha_5),
 \beta\alpha_1(\alpha_6+\alpha_7),\beta\alpha_1(-\alpha_8+\alpha_9),\notag\\
 &&\beta\alpha_1(\alpha_{10}-\alpha_{11}+\alpha_{12}),\beta\alpha^2_2,\beta\alpha_2(-\alpha_4+\alpha_5),-\beta\alpha_2\alpha_6,
 \beta\alpha_2\alpha_8,\beta\alpha_2(-\alpha_{10}+\alpha_{11}),\notag\\
 &&\beta\alpha_3\alpha_4,\beta\alpha_3\alpha_6,\beta\alpha_3\alpha_{10},\beta\bigl(\alpha_4(\alpha_4-\alpha_5)+
 \alpha_5(-\alpha_4+\alpha_5)\bigr),\beta\bigl(\alpha_6(\alpha_4-\alpha_5)+\alpha_4\alpha_7\bigr),\notag\\
 &&\beta\bigl(\alpha_8(-\alpha_4+\alpha_5)+\alpha_4\alpha_9\bigr),\beta\bigl(\alpha_{10} (\alpha_4-\alpha_5)+
 \alpha_{11}(-\alpha_4+\alpha_5)+\alpha_4\alpha_{12}\bigr),\notag\\
 &&\beta\bigl(\alpha_6(\alpha_6+\alpha_7)+\alpha_6\alpha_7\bigr),\beta(-\alpha_6\alpha_8+\alpha_6\alpha_9),
 \beta\bigl(\alpha_{10} (\alpha_6+\alpha_7)-\alpha_6\alpha_{11}+\alpha_6\alpha_{12}\bigr),\notag\\
 &&\beta\alpha^2_8,\beta\bigl(\alpha_{10}(-\alpha_8+\alpha_9)+\alpha_8\alpha_{11}\bigr),\notag\\
 &&\beta\bigl(\alpha_{10}(\alpha_{10}-\alpha_{11}+\alpha_{12})+\alpha_{11}(-\alpha_{10}+\alpha_{11})+\alpha_{10}\alpha_{12}
 \bigr)\Bigr\}.\notag
 \ee

 We can assume that $|\FF|\ne 4$, as our theorem holds for the field with four elements by
 Theorem~\ref{Th3.1}. Suppose also that $\FF\ne\FF_5$, so $|\FF|>5$. Let all $\alpha$'s be chosen to be~1.
 Then the constraints on $\beta$ are
 \[
 \beta\ne 0; \ 1+\beta\ne 0; \  1-\beta\ne 0; \ 1+2\beta\ne 0; \ 1+3\beta\ne 0,
 \]
 and they can be satisfied. It remains to consider the case $\FF=\FF_5$. In this case we let
 $\alpha_7=-1$ and all other $\alpha$'s be~1. A straightforward computation shows that the constraints
 on $\beta$ are now
 \[
 \beta\ne 0; \ 1+\beta\ne 0; \ 1-\beta\ne 0; \ 1+2\beta\ne 0,
 \]
 so $\beta=3$ works.

 \underline{Case~4.}~~We assume now that all off-diagonal entries of $A_{22}$ are nonzero.
 We pick $a\in \FF$ so that $a\ne 0$. By Theorem~\ref{Th3.1}, we may assume that the characteristic of
 $\FF\ne 2$. Let
 $$
 A_{22}=\frac{1}{2a}
 \begin{bmatrix}
 -1      & \ \ 1 & \ \ 1\\
 \ \; 1  & -1 & \ \ 1\\
  \ \; 1 & \ \ 1 & -1
 \end{bmatrix}.
 $$
 Then all off-diagonal entries of $A_{22}$ are nonzero, $A_{22}$ is invertible, and
 $$
 A_{22}^{-1}=
 \begin{bmatrix}
 0 & a & a\\
 a & 0 & a\\
 a & a & 0
 \end{bmatrix}.
 $$
 Hence,
 \[
 A_{12}A^{-1}_{22}=
 \begin{bmatrix}
 0 & 0 & 0\\
 0 & \alpha_1a\1 & \alpha_1a\1\\
 \alpha_2a\1 & 0 & \alpha_2a\1\\
 \alpha_3a\1 & \alpha_3a\1 & 0\\
 \alpha_5a\1 & \alpha_4a\1 & (\alpha_4+\alpha_5)a\1\\
 \alpha_7a\1 & (\alpha_6+\alpha_7)a\1 & \alpha_6a\1\\
 (\alpha_8+\alpha_9)a\1 & \alpha_9a\1 & \alpha_8a\1\\
 (\alpha_{11}+\alpha_{12})a\1 & (\alpha_{10}+\alpha_{12})a\1 & (\alpha_{10}+\alpha_{11})a\1
 \end{bmatrix},
 \]
 so all off-diagonal entries of $A_{12}A^{-1}_{22}A\trans_{12}$ are contained in
 \be
 &&\Bigl\{
 0,\alpha_1\alpha_2a,\alpha_1\alpha_3a,\alpha_1\alpha_5a,\alpha_1\alpha_7a,\alpha_1(\alpha_8+\alpha_9)a,
 \alpha_1(\alpha_{11}+\alpha_{12})a,\notag\\
 &&\alpha_2\alpha_3a,\alpha_2\alpha_4a,\alpha_2(\alpha_6+\alpha_7)a,\alpha_2\alpha_9a,
 \alpha_2(\alpha_{10}+\alpha_{12})a,\alpha_3(\alpha_4+\alpha_5)a,\alpha_3\alpha_6a,\notag\\
 &&\alpha_3\alpha_8a,\alpha_3(\alpha_{10}+\alpha_{11})a,2\alpha_4\alpha_5a,
 \bigl(\alpha_5\alpha_6+(\alpha_4+\alpha_5)\alpha_7\bigr)a,\bigl(\alpha_4\alpha_8+(\alpha_4+\alpha_5)\alpha_9 \bigr)a, \notag\\
 &&\bigl(\alpha_5\alpha_{10}+\alpha_4\alpha_{11}+(\alpha_4+\alpha_5)\alpha_{12}\bigr)a,2\alpha_6\alpha_7a,
 \bigl((\alpha_6+\alpha_7)\alpha_8+\alpha_6\alpha_9 \bigr)a,\notag\\
 &&\bigl(\alpha_7\alpha_{10}+(\alpha_6+\alpha_7)\alpha_{11}+\alpha_6\alpha_{12} \bigr)a,2\alpha_8\alpha_9a,
 \bigl((\alpha_8+\alpha_9)\alpha_{10}+\alpha_9\alpha_{11}+\alpha_8\alpha_{12} \bigr)a,\notag\\
 &&\bigl((\alpha_{11}+\alpha_{12})\alpha_{10}+(\alpha_{10}+\alpha_{12})\alpha_{11}+(\alpha_{10}+\alpha_{11})\alpha_{12} \bigr)a\Bigr\}.\notag
 \ee

 Suppose first that $\FF\ne\FF_5$, so $|\FF|\ge 7$. Letting all $\alpha$'s be~1 we obtain the constraints on $a$:
 \[
 a\ne 0; \ 1+a\ne 0;  \ 1+2a\ne 0; \ 1+3a\ne 0; \ 1+4a\ne 0; \ 1+6a\ne 0.
 \]
 These constraints can be satisfied in any finite field with at least seven elements.

 In case $\FF=\FF_5$, let $\alpha_1=\alpha_2=\alpha_3=\alpha_4=\alpha_6=\alpha_8=\alpha_{10}=1, \
 \alpha_5=\alpha_7=\alpha_9=\alpha_{11}=-1$ and $\alpha_{12}=2$. This choice yields the constraints
 \[
 a\ne 0, \ 1+a\ne 0; \ 1-a\ne 0; \ 1+3a\ne 0,
 \]
 so the choice $a=2$ works.
 \end{proof}

 \end{document}